\documentclass[12pt]{article}
\usepackage{amssymb}
\makeatletter\@addtoreset {equation}{section}\makeatother

\newtheorem{theorem}{Theorem}

\newtheorem{remark}{Remark}
\newtheorem{corollary}{Corollary}

\newenvironment{proof}{
    \noindent {\it Proof.}}{\hfill$\Box$
}

\usepackage[dvips]{epsfig}
\usepackage{graphicx}

\begin{document}

\title{\bf Factorization of the Indefinite Convection-Diffusion Operator.}

\author{Marina Chugunova \\ {\small Department of Mathematics, University of
Toronto, Canada} \\ Vladimir Strauss \\ {\small Department of Pure
\& Applied Mathematics, Simon Bolivar University, Venezuela} }

\date{\today}
\maketitle

\begin{abstract}
We prove that some non-self-adjoint differential operator admits
factorization and apply this new representation of the operator to
construct explicitly its domain. We also show that this operator is
J-self-adjoint in some Krein space.

MSC classes: 34Lxx; 76Rxx; 34B24

\end{abstract}

\section{Introduction}

The time-evolution of a thin film of viscous fluid on the inner
surface of a rotating cylinder can be approximated by
 the Cauchy problem for the periodic backward-forward heat equation
\begin{equation}
\label{heat PDE}
 y_t + l[y] = 0, \quad y(0,x) = y_0, \quad y(x,t) = y(x + 2 \pi,t), \quad x\in (-\pi,\pi), \quad t > 0
\end{equation}
where
\begin{equation}
\label{heat operator}
 l[y] = \epsilon (\sin x y_x)_x + y_x, \quad \epsilon > 0.
\end{equation}

This model was derived by Benilov, O'Brien and Sazonov
\cite{Benilov1, Benilov2} under assumption that the parameter
$\epsilon$, related to the thickness of the film, is sufficiently
small, i.e $\epsilon << 1$.

It was first shown numerically by Benilov, O'Brien and Sazonov
\cite{Benilov1} that the spectrum of the operator $L$ defined by the
operation $l[.]$ and periodic boundary conditions $y(-\pi) = y(\pi)$
consists of pure imaginary eigenvalues only. This result was quite a
surprising because it is well known that the Cauchy problem
\ref{heat PDE} is ill-posed at least for the class of finitely
smooth functions.

The spectral properties of the operator $L$ were studied rigorously
by Davies \cite{Davies} and by Chugunova, Pelinovsky \cite{ChugPel}.
Approaching the problem in two different ways they proved
analytically that if the parameter $\epsilon$ is within the interval
$|\epsilon| < 2$ then the operator $L$ admits closure in the Hilbert
space $L^2[-\pi, \pi]$ and being restricted to the orthogonal to a
constant subspace has compact inverse of Hilbert-Schmidt type and as
a consequence its spectrum is discrete with the only accumulation
point at infinity.

The numerical conjecture that all eigenvalues are pure imaginary was
recently proved by Weir \cite{Weir}. The elegant proof based on the
continuation of the eigenfunctions into Hardy space on the unit disk
and explicit construction of the symmetric operator.

Our goal in this paper is to find the factorization of the operator
$L$ and to construct its domain explicitly. As a consequence we
prove that the operator $L$ acting in physical space is unitary
equivalent to the operator $A$ introduced by Davies \cite{Davies}
acting in the Fourier space. We also prove that the non-self-adjoint
differential operator $L$ belongs to the class of $J$-self-adjoint
operators in some Krein space.  For basic facts related to Krein
spaces see \cite{AI}.

\section{Factorization of the non-self-adjoint operator $L$.}

We denote by $\mathfrak{D}(T)$ and $\mathfrak{R}(T)$ the domain and
the range of linear operator $T$ respectively. The notation
$\mathcal{L}^2$ is used for the standard Lebesgue space of scalar
functions defined on the interval $(-\pi , \pi)$. From here on $L$
is the indefinite convection-diffusion operator $L:$
$$(Ly)(x)=\epsilon\cdot (\sin(x)y^{\prime}(x))^{\prime}+ y^{\prime}(x), \quad \mathcal{L}^2\mapsto \mathcal{L}^2$$
with the domain of all absolutely continuous $2\pi$-periodic
functions $y(x)$ such that $(Ly)(x)\in \mathcal{L}^2$.

In addition, we define the operator $S:$
$$\mathcal{L}^2\mapsto \mathcal{L}^2, \quad (Sy)(x)=y^{\prime}(x),$$
where $y^{\prime}(x)\in \mathcal{L}^2$, $y(-\pi )=y(\pi )$, and the
operator $M:$
$$\mathcal{L}^2\mapsto \mathcal{L}^2, \quad (My)(x)=\epsilon\cdot(\sin(x)y(x))^{\prime}+ y(x)$$
with the domain of all absolutely continuous  functions $y(x)$ such
that $(My)(x)\in \mathcal{L}^2$.

\begin{theorem}\label{main}If the parameter $\epsilon\in (0,\, 2)$, then $L=MS$.\end{theorem}

\begin{proof}
Let us consider the operator $A:$ $$\mathcal{L}^2\mapsto
\mathcal{L}^2, \quad (Ay)(x)=(\sin(x)y(x))^{\prime}$$ with
$\mathfrak{D}(A)=\{y(x)\ |\ y(x), (Ay)(x)\in \mathcal{L}^2\}$. Then
a function $y(x)$ can be written as
\begin{equation}\label{1} y(x)=\frac{1}{\sin(x)}\cdot \big(
c+\int_0^x \theta(t)dt\big) ,\quad \theta(t)\in
\mathcal{L}^2.\end{equation} If $x>0$ then
$$ |\int_0^x
\theta(t)dt|\leq \frac{1}{\sin(x)}\cdot \alpha(x)\cdot x^{1/2},$$
where $\alpha(x)=\big(\int_0^x |\theta(x)|^2dx\big)^{1/2}$. Since
the two summands in (\ref{1}) have different orders of growth as
$x\to 0$ this implies that if $y(x)\in \mathcal{L}^2$ then $c=0$ and
\begin{equation}\label{2} y(x)=\frac{1}{\sin(x)}\cdot
\int_0^x \theta(t)dt .\end{equation} Moreover,
\begin{equation}\label{3} |y(x)|\leq\frac{x^{1/2}}{\sin(x)}\cdot
\alpha(x) .\end{equation} A small modification of the same reasoning
leads to the following estimation for every $x\in (-\pi ,\pi)$
\begin{equation}\label{4} |y(x)|\leq\frac{|x|^{1/2}}{|\sin(x)|}\cdot
\alpha(x) .\end{equation} with $\alpha(x)=\big|\int_0^x
|\theta(x)|^2dx\big|^{1/2}$.

Alternatively the same function $y(x)$ can be written as
\begin{equation}\label{5} y(x)=\frac{1}{\sin(x)}\cdot \big(
\tilde{c}-\int_x^{\pi} \theta(t)dt\big) ,\quad \theta(t)\in
\mathcal{L}^2.\end{equation} with the same  $\theta(x)$ as in
(\ref{1}). Representation (\ref{5}) yields the following relations
\begin{equation}\label{6} y(x)=\frac{-1}{\sin(x)}\cdot
\int_x^{\pi} \theta(t)dt \end{equation} and
\begin{equation}\label{7} |y(x)|\leq\frac{(\pi -x)^{1/2}}{\sin(x)}\cdot
\beta(x) \end{equation} with $\beta(x)=\big|\int_x^{\pi}
|\theta(x)|^2dx\big|^{1/2}$.

It follows from (\ref{2}) and (\ref{6}) that
\begin{equation}\label{8} \int_{0}^{\pi}\theta(t)dt=0.\end{equation}

Starting from the point $-\pi$ one can also obtain that
\begin{equation}\label{9} y(x)=\frac{1}{\sin(x)}\cdot \int_{-\pi}^x
\theta(t)dt,\end{equation}
\begin{equation}\label{10} |y(x)|\leq\frac{(-\pi +x)^{1/2}}{|\sin(x)|}\cdot
\gamma(x) \end{equation} with $\gamma(x)=\big(\int_{-\pi}^x
|\theta(x)|^2dx\big|^{1/2}$ and
\begin{equation}\label{11} \int_{-\pi}^{0}\theta(t)dt=0.\end{equation}
Obviously, the natural domain of the operator $B:$
$$\mathcal{L}^2\mapsto \mathcal{L}^2, \quad (By)(x)=\sin(x)\cdot
y(x)^{\prime},$$ where $y(x), (By)(x)\in \mathcal{L}^2$ coincides
with the domain of the operator $A$. And as a consequence we obtain
that $A^{*}=-B$.

Now let us define the operator $C:$ $\mathcal{L}^2\mapsto
\mathcal{L}^2$, $$(Cy)(x)=-i\cdot\big(\sin(x)\cdot
y(x)^{\prime}-\frac{1}{2}\cos(x)y(x)\big)=-i\cdot\big((By)(x)+\frac{1}{2}\cos(x)y(x)\big),$$
where $y(x), (Cy)(x)\in \mathcal{L}^2$ and $D:$
$$(Dy)(x)=(\frac{\epsilon}{2}\cdot\cos(x)+1)\cdot y(x), \quad \mathcal{L}^2\mapsto \mathcal{L}^2. $$
It follows directly from the relation between the operators $A$ and
$B$ and inequalities (\ref{4}), (\ref{7}, (\ref{10}) that $C$ is a
self-adjoint operator.

We now restrict the parameter $\epsilon$ to the interval $0<\epsilon
<2$. This implies that $D$ has the bounded inverse, and hence that
$D^{-1/2}\cdot C\cdot D^{-1/2}$ is a self-adjoint operator.
Therefore the operator defined as
$$\big( i\cdot\epsilon\cdot D^{-1/2}\cdot C\cdot D^{-1/2}+I\big)$$
also has the bounded inverse, and the same is true for the operator
$$M=\big( i\cdot\epsilon\cdot C + D\big)=\epsilon\cdot A+I.$$
Finally, let us show that the subspace $\{const \}^{\perp}\cap
\mathcal{L}^2$ is invariant under $M$ and $M^{-1}$. Indeed, if
$y(x)$ belongs to the domain of $M$, then $y(x)$ has representation
(\ref{2}), and hence
$$(My)(x) =y(x)+\epsilon\cdot \theta(x).$$ It follows by (\ref{8})
and (\ref{11}) that
$$\int_{-\pi}^{\pi}(My)(x)dx=\int_{-\pi}^{\pi}y(x)dx,$$ so
$(My)(x)\in\{const \}^{\perp}\cap \mathcal{L}^2$ if and only if
$y(x)\in\{const \}^{\perp}\cap \mathcal{L}^2$.
\par

Our main goal is to show that $L=MS$. It is easy to check that a
function $y(t)$ belongs to the domain (within $\mathcal{L}^2$) of
$L$ only if
\begin{equation}\label{12}y^{\prime}(x)=(\tan(|x|/2))^{-1/\epsilon}\sin^{-1}(x)\cdot
\big(c+\int_0^x (\tan(|t|/2))^{1/\epsilon}\phi(t)
dt\big),\end{equation}
where $\phi(t)\in \mathcal{L}^2$. Let $x>0$.
Then
 $$|\int_0^x (\tan(t/2))^{1/\epsilon}\phi(t) dt|\leq
\big(\int_0^x t^{2/\epsilon}dt\big)^{(1/2)}\cdot  \big(\int_0^x
|\phi(t)(\frac{\tan(t/2)}{t})^{1/\epsilon}|^2dt\big)^{(1/2)}.$$ The
latter estimation yields $$|\int_0^x (\tan(t/2))^{1/\epsilon}\phi(t)
dt|\leq x^{1/2+1/\epsilon}\cdot \alpha(x),$$ where $\alpha(x)\to 0$
if $x\to 0$. Then
$$|(\tan(x/2))^{-1/\epsilon}\sin^{-1}(x)\cdot
\big(\int_0^x (\tan(t/2))^{1/\epsilon}\phi(t) dt\big)|\leq
x^{-1/2}\cdot \beta(x),$$ where $\beta(x)\to 0$ if $x\to 0$. If
$c=0$, we obtain that
$$|y(x)-y(0)| = |\int_0^x(\tan(\tau/2))^{-1/\epsilon}\sin^{-1}(\tau)\cdot
\big(\int_0^{\tau} (\tan(t/2))^{1/\epsilon}\phi(t) dt\big)d\tau|\leq
\gamma (x)x^{1/2},$$ where $\gamma(x)\to 0$ if $x\to 0$. Thus, if
$c=0$, $y(x)$ is continuous at zero. At the same time if $c\not= 0$,
$y(x)$ contains an additional summand of the order $x^{-1/\epsilon}$
that is out of $\mathcal{L}^2$ for $\epsilon\leq 2$. Thus, $c=0$ in
(\ref{12}) and so

\begin{equation}\label{13}y^{\prime}(x)=(\tan(|x|/2))^{-1/\epsilon}\sin^{-1}(x)\cdot
\int_0^x (\tan(|t|/2))^{1/\epsilon}\phi(t) dt\, .\end{equation} This
representation was derived under the hypothesis that $x>0$, but it
is clear that it is valid for every $x\in (-\pi ,\pi )$.
\par
Now let us assume that there is $y(x)$ such that
\begin{equation}y(x)\in \mathfrak{D}(L)\ \mbox{ but }\ y(x)\not\in
\mathfrak{D}(MS).\label{14}\end{equation} Then there are two
options.
\begin{itemize}
\item $(Ly)(x)\not\in \mathfrak{R}(MS)=\{const \}^{\perp}\cap \mathcal{L}^2$;
\item $(Ly)(x)\in \mathfrak{R}(MS)$.
\end{itemize}
The first option means that $\mathfrak{R}(L)=\mathcal{L}^2$ and
without loss of generality one can assume that $(Ly)(x)\equiv 1$.
Then by (\ref{13})
$$y^{\prime}(x)=(\tan(|x|/2))^{-1/\epsilon}\sin^{-1}(x)\cdot \int_0^x
(\tan(|t|/2))^{1/\epsilon} dt\, .$$ Thus, $y^{\prime}(x)>0$ for
every $x\in (-\pi ,\pi )$. The latter is impossible for absolutely
continuous $2\pi$-periodic function. It is a contradiction!
\par
Now let us consider the second option. If $y(x)$ satisfies
(\ref{14}) and $(Ly)(x)\in \mathfrak{R}(MS)$, then there is $z(x)\in
\mathfrak{D}(MS)$ such that $(Lz)(x)=(Ly)(x)$. The latter yields
$(L(z-y))(x)\equiv 0$. Thus, by virtue of (\ref{13}),
$y(x)=z(x)+const$ that is impossible thanks to
(\ref{14}).\end{proof}
\begin{corollary}\label{cor1} $L$ is a closed operator with the
non-empty resolvent set and its resolvent has Sturm-Liouville
property.\end{corollary} Let us define by  $L_0$ the operator that
represents the restriction of $L$ on the set of all $2\pi$-periodic
smooth functions and by $\bar{L}_0$ its closure.
\begin{theorem} \label{Lbar} $L=\bar{L}_0$. The operator $L$ is $J$-self-adjoint
in the Krein space with indefinite metric $J$ defined as $J(f(x)) =
f(\pi - x)$.
\end{theorem}
\begin{proof} It is evident that $\bar{L}_0\subseteq L$. In the paper
of Davies \cite{Davies} was shown that $\bar{L}_0$ has the non-empty
resolvent set. If $\bar{L}_0\not= L$ then every $\lambda\in
\rho(\bar{L}_0)$ belongs to $\sigma_p(L)$ but it is impossible due
to Corollary \ref{cor1}.

The adjoint operator is defined by the operation
$$(L^*y)(x)=\epsilon\cdot(\sin(x)y^{\prime}(x))^{\prime}- y^{\prime}(x)$$
on the same domain as the operator $L$. The last statement of the
theorem follows immediately from the equality  $L = JL^{*}J$.
\end{proof}

\begin{remark} The J-self-adjoint operator $L$ being restricted to the
subspace orthogonal to a constant has a compact inverse that implies
that there exist two, non-positive and non-negative with respect to
the metric $J$, maximal invariant under the operator $L$
subspaces.\end{remark}

\section{Domain of the operator $L$}

As it was proved in the previous section, the operator $M:$
$$\mathcal{L}^2\mapsto \mathcal{L}^2, \quad (My)(x)=\epsilon\cdot
(\sin(x)y(x))^{\prime}+ y(x)$$ with the domain of all absolutely
continuous  functions $y(x)$ such that $(My)(x)\in \mathcal{L}^2$
has bounded inverse. Hence, there is a constant $p_1$ such that
\begin{equation}\label{add1}
\|\epsilon\cdot (\sin(x)y(x))^{\prime}+ y(x)\|_{L^2} \geq p_1\cdot
\|y(x)\|_{L^2}\end{equation} for every $y(x)\in \mathfrak{D}(M)$.
Next, the operator $L$ is closed, so its domain is closed with
respect to the norm of the graphic of $L$, i.e. for the norm
\begin{equation}\label{add2}\|y(x)\|_{g}:=\{\|y(x)\|^2_{L^2}+\|\epsilon\cdot
(\sin(x)y^{\prime}(x))^{\prime}+
y^{\prime}(x)\|^2_{L^2}\}^{1/2}.\end{equation} Our aim is to show
that for $y(x)\in\mathfrak{D}(L)\cap \{const\}^{\perp}$ the norm
(\ref{add2}) is equivalent to the following norm
\begin{equation}\label{add3}\|y(x)\|_{m}:=\{\|y^{\prime}(x)\|^2_{L^2}+\|\sin(x)\cdot
y^{\prime}(x)\|^2_{L^2} +\|(\sin(x)\cdot
y^{\prime}(x))^{\prime}\|^2_{L^2}\}^{1/2}.\end{equation} Indeed, for
every $y(x)\in\mathfrak{D}(L)\cap \{const\}^{\perp}$ we have
$$\|y(x)\|_{g}\leq \{\|y^{\prime}(x)\|^2_{L^2}+2\|\epsilon\cdot
(\sin(x)y^{\prime}(x))^{\prime}\|_{L^2}+2\cdot \|
y^{\prime}(x)\|^2_{L^2}\}^{1/2}\leq p_2\cdot\|y(x)\|_{m},$$ where
$p_2=\max \{ 3,\, 2\epsilon \}$. From the other hand, taking into
account (\ref{add1}), we have
$$\|y(x)\|_{g}\geq \|\epsilon\cdot
(\sin(x)y^{\prime}(x))^{\prime}+ y^{\prime}(x)\|_{L^2}\geq $$
$$\frac{p_1}{2}\cdot\|y^{\prime}(x)\|_{L^2}+ \frac{1}{2}\cdot
\|\epsilon\cdot (\sin(x)y^{\prime}(x))^{\prime}+
y^{\prime}(x)\|_{L^2}\geq $$
$${3p_3}\cdot\|y^{\prime}(x)\|_{L^2}+ {p_3}\cdot
\|\epsilon\cdot (\sin(x)y^{\prime}(x))^{\prime}+
y^{\prime}(x)\|_{L^2} \geq$$
$${p_3}\cdot\big( \|y^{\prime}(x)\|_{L^2}+ \|\sin(x)y^{\prime}(x)\|_{L^2}+
\|\epsilon\cdot (\sin(x)\cdot y^{\prime}(x))^{\prime}\|_{L^2}\big)
,$$ where $p_3=\min \{\frac{p_1}{6},\frac{1}{2}\}$.

It follows from above that the domain of $L$ is the linear
sub-manifold  $H$ of the Sobolev space $H^1 (-\pi, \pi)$:
$$ \mathfrak{D}(L) = H : { f \in H^1 (-\pi, \pi),\quad f(\pi) = f(-\pi,) \quad
\sin(x)f' \in H^1 (-\pi, \pi) }$$ and is a Hilbert space with the
norm defined as:
$$||f||^2 = || f'||^2_{L^2} + || \sin(x)f'(x)||^2_{H^1}.$$
\begin{remark}
As a consequence of Theorem \ref{main}, the tridiagonal matrix
operator $\mathcal{A}$ in Fourier space \cite{ChugPel}.
\begin{equation}
\label{matrix-A}
\mathcal{A} = \left[ \begin{array}{ccccc} 1 & \epsilon & 0 & 0 & \cdots \\
-\epsilon & 2 & 3 \epsilon & 0 & \cdots \\ 0 & -3 \epsilon & 3 &
6 \epsilon & \cdots \\ 0 & 0 & -6 \epsilon & 4 & \cdots \\
\vdots & \vdots & \vdots & \vdots & \ddots \end{array} \right]
\end{equation}
admits the same type of the factorization as in Theorem \ref{main}
\begin{equation}
\label{matrix-A1}
\mathcal{A} = \mathcal{B}\mathcal{C} = \left[ \begin{array}{ccccc} 1 & \epsilon/2 & 0 & 0 & \cdots \\
-\epsilon & 1 &  \epsilon & 0 & \cdots \\ 0 & -3/2 \epsilon & 1 &
3/2 \epsilon & \cdots \\ 0 & 0 & -2 \epsilon & 1 & \cdots \\
\vdots & \vdots & \vdots & \vdots & \ddots \end{array} \right]
\left[ \begin{array}{ccccc} 1 & 0 & 0 & 0 & \cdots \\
0 & 2 & 0 & 0 & \cdots \\ 0 & 0 & 3 & 0& \cdots \\ 0 & 0 & 0 & 4 & \cdots \\
\vdots & \vdots & \vdots & \vdots & \ddots \end{array} \right]
\end{equation}
there $\mathcal{B}$ has bounded inverse and $\mathcal{C}$ has
inverse of the Hilbert-Schmidt type. The operator $\mathcal{A}$ is
$\mathcal{J}$-self-adjoint with $\mathcal{J} = diag(-1,1,-1,1, ...
)$. Thus, the natural domain $\mathfrak{D}(\mathcal{A})$ is the
Hilbert space of all number sequences $\{ f_n \}_1^{\infty}$ with
the norm
$$\|\{ f_n \}_1^{\infty}\|:=
\big\{
\sum_{n=1}^{\infty}n^2\big(|f_n|^2+|(n+1)f_{n+1}-(n-1)f_{n-1}|^2\big)\big\}^{1/2}
.$$ In the latter formula we put for definiteness $f_0=0$. Note also
that there exist two, non-positive and non-negative with respect to
the indefinite metric defined by $\mathcal{J}$, maximal invariant
under the operator $\mathcal{A}$, subspaces.
\end{remark}

{\bf Acknowledgement.}  M.C. is supported by the NSERC Postdoctoral
Fellowship.

\end{document}